\documentclass[leqno,10pt]{amsart}

\usepackage{amsmath,amssymb,mathrsfs} 
\usepackage[%dvipdfmx, 
colorlinks,linkcolor=blue,citecolor=blue,urlcolor=blue]{hyperref}
\usepackage{times}
\usepackage{gensymb}
\usepackage[all]{xy}

\theoremstyle{plain}
    \newtheorem{thm}{Theorem}[section]
    
    \newtheorem{lem}[thm]{Lemma}
    \newtheorem{cor}[thm]{Corollary}
\theoremstyle{definition}

\theoremstyle{remark}

\numberwithin{equation}{section}

\allowdisplaybreaks

\def\Hom{\operatorname{Hom}}
\def\C{\mathbb{C}}\def\Q{\mathbb{Q}}\def\F{\mathbb{F}}\def\N{\mathbb{N}}
\def\ol#1{\overline{#1}}\def\wh#1{\widehat{#1}}

\def\a{\alpha}\def\b{\beta}\def\g{\gamma}\def\d{\delta}\def\e{\varepsilon}
\def\l{\lambda}\def\k{\kappa}

\def\G{\Gamma}

\def\r{\rho}

\newcommand{\baa}{\boldsymbol\alpha}
\newcommand{\bbb}{\boldsymbol\beta}
\newcommand{\bgg}{\boldsymbol\gamma}

\newcommand\n{\nu}\newcommand\m{\mu}
\newcommand\vp{\varphi}

\newcommand\ck{\wh{\k^*}}
\newcommand\0{\circ}

\begin{document}

\title{Product formulas for  hypergeometric functions over finite fields}
\author{Noriyuki Otsubo \and Takato Senoue}

\address{Department of Mathematics and Informatics, Chiba University, Inage, Chiba, 263-8522 Japan}
\email{otsubo@math.s.chiba-u.ac.jp}
\address{Graduate School of Science and Engineering, Chiba University, Inage, Chiba, 263-8522 Japan}
\email{takas0609@gmail.com}

\begin{abstract}
Many product formulas are known classically for generalized hypergeometric functions over the complex numbers. 
In this paper, we establish some analogous formulas for generalized hypergeometric functions over finite fields. 
\end{abstract}

\date{\today.}

\maketitle
%%%%%%%%%%%%%%%%%%%%%%%%%

\section{Introduction}

A generalized hypergeometric function ${}_rF_s(x)$ over the complex numbers is defined by the power series 
$${}_rF_s\left({a_1,\dots, a_r \atop b_1,\dots, b_s};x\right)
=\sum_{n=0}^\infty \frac{\prod_{i=1}^r (a_i)_n}{(1)_n \prod_{i=1}^s (b_i)_n} x^n, $$
where $a_i$ and $b_i$ are complex parameters with $-b_i\not\in\N$, and 
$$(a)_n=\G(a+n)/\G(a)=a(a+1)\cdots(a+n-1)$$
is the Pochhammer symbol (cf. \cite{bailey-book}). 

Over a finite field, there are several definitions of analogous functions, due to 
Koblitz \cite{koblitz}, Greene \cite{greene}, McCarthy \cite{mccarthy}, Fuselier--Long--Ramakrishna--Swisher--Tu \cite{fuselieretal} when $r=s+1$, 
and to Katz \cite{katz} and the first author \cite{otsubo} in general. Here we use the last one. 
Such a function is a map from a finite field $\k$ to $\C$, where the parameters are multiplicative characters of $\k$. 
See Section 2 for the definition and \cite[Remark 2.12]{otsubo} for the relation with other definitions.  

Over the complex numbers, there are many product formulas expressing a product of two hypergeometric functions in terms of another one. 
Bailey \cite[(2.01)--(2.12)]{bailey1} gave the following list of twelve formulas.   
Here, all but \eqref{b6} involve ${}_rF_s$-functions with $r \ne s+1$: 
note that ${}_0F_0(x)=e^x$ and ${}_1F_0\left({a\atop};x\right)=(1-x)^{-a}$. 
\begin{equation}\label{b1}
e^{-x} {}_1F_1\left({a \atop b};x\right)={}_1F_1\left({b-a \atop b};-x\right). 
\end{equation}
\begin{equation}\label{b2}
e^{-\frac{x}{2}} {}_1F_1\left({a \atop 2a};x\right) ={}_0F_1\left({\atop  a+\frac{1}{2} };\frac{x^2}{16}\right). 
\end{equation}
\begin{equation}\label{b3}
{}_0F_1\left({\atop  2a };x\right) {}_0F_1\left({\atop 2b};x \right)
={}_2F_3\left({a+b,a+b-\frac{1}{2} \atop 2a,2b,2a+2b-1};4x\right).
\end{equation}
\begin{equation}\label{b4}
{}_0F_1\left({\atop  2a };x\right) {}_0F_1\left({\atop 2a};-x \right)
={}_0F_3\left({\atop 2a,a,a+\frac{1}{2}};-\frac{x^2}{4}\right).
\end{equation}
\begin{equation}\label{b5}
\begin{split}
&{}_0F_1\left({\atop  2a };x\right) {}_0F_1\left({\atop 2-2a};-x \right)
\\&
={}_0F_3\left({\atop \frac{1}{2},a+\frac{1}{2}, \frac{3}{2}-a};-\frac{x^2}{4}\right)
+\frac{1-2a}{2a(1-a)} x {}_0F_3\left({\atop \frac{3}{2},a+1,2-a};-\frac{x^2}{4}\right
).
\end{split}
\end{equation}
\begin{equation}\label{b6}
(1-x)^{a+b-c} {}_2F_1\left({a,b\atop c};x\right)={}_2F_1\left({c-a,c-b \atop c};x\right).
\end{equation}
\begin{equation}\label{b7}
{}_2F_0\left({2a,2b \atop};x\right) {}_2F_0\left({2a,2b\atop};-x\right)={}_4F_1\left({2a,2b,a+b,a+b+\frac{1}{2} \atop 2a+2b};4x^2\right).
\end{equation}
\begin{equation}\label{b8}\begin{split}
& {}_2F_0\left({2a,1-2a \atop};x\right) {}_2F_0\left({2b,1-2b \atop};-x\right) 
 \\& ={}_4F_1\left({a-b+\frac{1}{2}, b-a+\frac{1}{2}, a+b, 1-a-b \atop \frac{1}{2}};4x^2\right) 
\\& -(2a-2b)(2a+2b-1) x {}_4F_1\left({a-b+1, b-a+1, a+b+\frac{1}{2}, \frac{3}{2}-a-b \atop \frac{3}{2}};4x^2\right). 
\end{split}
\end{equation}
\begin{equation}\label{b9}
{}_1F_1\left({a\atop 2b};x\right) {}_1F_1\left({a\atop 2b};-x\right) ={}_2F_3\left({a,2b-a \atop 2b,b,b+\frac{1}{2}};\frac{x^2}{4}\right).
\end{equation}
\begin{equation}\label{b10}\begin{split}
& {}_1F_1\left({a\atop 2b};x\right) {}_1F_1\left({a-2b+1\atop 2-2b};-x\right) 
 =  {}_2F_3\left({a-b+\frac{1}{2},b-a+\frac{1}{2} \atop \frac{1}{2}, b+\frac{1}{2},\frac{3}{2}-b};\frac{x^2}{4}\right)
\\& +\frac{(a-b)(1-2b)}{2b(1-b)} x {}_2F_3\left({a-b+1,b-a+1 \atop \frac{3}{2}, b+1,2-b};\frac{x^2}{4}\right). 
\end{split}\end{equation}
\begin{equation}\label{b11}
{}_1F_1\left({2a\atop 4a};x\right) {}_1F_1\left({2b\atop 4b};-x\right) 
={}_2F_3\left({a+b, a+b+\frac{1}{2} \atop 2a+\frac{1}{2}, 2b+\frac{1}{2},2a+2b};\frac{x^2}{4}\right). 
\end{equation}
\begin{equation}\label{b12}\begin{split}
& {}_0F_2\left({\atop 6a,6b};x\right){}_0F_2\left({\atop 6a,6b};-x\right)
\\& ={}_3F_8\left({2a+2b,2a+2b+\frac{1}{3},2a+2b-\frac{1}{3} \atop 6a,6b,3a,3a+\frac{1}{2},3b,3b+\frac{1}{2}, 3a+3b, 3a+3b-\frac{1}{2}};-\frac{27x^2}{64}\right).
\end{split}\end{equation}
Originally, the formulas \eqref{b1}, \eqref{b2} are due to Kummer,  
\eqref{b6} is due to Euler, \eqref{b9}, \eqref{b12} are due to Ramanujan, and \eqref{b10} and a special case of \eqref{b11} are due to Preece (see \cite{bailey1}). 
 
Over a finite field, an analogue of \eqref{b6} was given by Greene \cite[Theorem 4.4 (iv)]{greene} 
and analogues of \eqref{b1}, \eqref{b2}, \eqref{b9} were given by the first author \cite{otsubo} (see Section 3 of this paper). 
In this paper, we give the remaining eight formulas. 
As well as the complex case, our strategy is to compare the ``coefficients", i.e. the Fourier transforms. 
Then, the proofs reduce to formulas on special values ${}_2F_1(\pm 1)$ and ${}_3F_2(1)$, 
which are finite analogues of summation formulas of Euler--Gauss, Kummer and Dixon. 
 
For a product of Gauss hypergeometric functions ${}_2F_1(x)$, 
we have classically Clausen's formula (cf. \cite[10.1 (4)]{bailey-book}) 
\begin{equation}\label{b13}
{}_2F_1\left({a,b\atop a+b+\frac{1}{2}};x\right)^2 = {}_3F_2\left({2a,2b,a+b\atop 2a+2b,a+b+\frac{1}{2}};x\right), 
\end{equation}
and Bailey's formula \cite[(6.1)]{bailey2} 
\begin{equation}\label{b14}\begin{split}
&{}_2F_1\left({2a,2b\atop c};x\right){}_2F_1\left({2a,2b\atop 2a+2b-c+1};x\right)  
\\&= {}_4F_3\left({2a,2b,a+b,a+b+\frac{1}{2} \atop 2a+2b, c,2a+2b-c+1};4x(1-x)\right).
\end{split}\end{equation}
A finite analogue of \eqref{b13} was first proved by Evans--Greene \cite[Theorem 1.5]{evans2} 
and the first author gave another proof  \cite[Theorem 6.5]{otsubo}. 
A finite analogue of \eqref{b14} was given recently by Barman--Tripathi \cite[Theorem 1.1]{tripathi-barman}. 
In Section 4, we give a short proof of the last formula (Corollary \ref{b-t}). 

%%%
\section{Hypergeometric functions over finite fields}

Let us recall from \cite{otsubo} the definition and some properties of hypergeometric functions over finite fields. 
Let $\k$ be a finite field of characteristic $p$ with $q$ elements. 
Fix a non-trivial additive character $\psi \in \Hom(\k,\C^*)$. For a multiplicative character $\vp \in \ck :=\Hom(\k^*,\C^*)$, we set $\vp(0)=0$. 
Let $\e\in\ck$ denote the unit character, and $\d \colon \ck \to \C$ the characteristic function of $\e$, i.e. $\d(\e)=1$ and $\d(\vp)=0$ if $\vp \ne \e$. 
The Gauss sum, a finite field analogue of the gamma function, and its variant are defined for $\vp \in \ck$  by 
$$g(\vp)=-\sum_{x\in\k} \psi(x)\vp(x), \quad g^\0(\vp)=q^{\d(\vp)}g(\vp).$$
Note that $g(\e)=1$, hence $g^\0(\e)=q$. 
For $\a, \n \in \ck$, define the Pochhammer symbol and its variant by 
$$(\a)_\n=g(\a\n)/g(\a), \quad (\a)^\0_\n=g^\0(\a\n)/g^\0(\a).$$
We have evidently, 
\begin{equation}\label{g1} 
(\a)_{\b\n}=(\a)_\b(\a\b)_{\n}, \quad 
(\a)^\circ_{\b\n}=(\a)^\circ_\b(\a\b)^\circ_{\n}.
\end{equation}
An analogue of the reflection formula $\G(s)\G(1-s)={\pi}/{\sin \pi s}$ is
\begin{equation}\label{g2}
g(\vp)g^\0(\ol\vp)=q \vp(-1),
\end{equation}
where $\ol\vp=\vp^{-1}$ (cf. \cite[Proposition 2.2]{otsubo}), 
which implies 
\begin{equation}\label{g3}
(\a)_\n(\ol\a)^\0_{\ol\n}=\n(-1).
\end{equation}
If $n \mid q-1$, we have the Davenport--Hasse multiplication formula
\begin{equation}\label{g4}
g(\a^n)=\a^n(n)\prod_{\vp\in\ck, \vp^n=\e} \frac{g(\a\vp)}{g(\vp)}, 
\end{equation}
(cf. \cite[Theorem 3.10]{otsubo}),  
which implies 
\begin{equation}\label{g5}
(\a^n)_{\n^n}=\n^n(n) \prod_{\vp^n=\e} (\a\vp)_\n, \quad 
(\a^n)^\0_{\n^n}=\n^n(n)\prod_{\vp^n=\e} (\a\vp)^\0_\n.
\end{equation}

The parameter set, denoted by $P$, is the free abelian monoid over $\ck$. 
Let $\deg\colon P\to \N$ be the degree map. 
Let $(\ , \ ) \colon P\times P \to \N$ be the symmetric pairing extending 
$(\a,\b)=\d(\a\ol\b)$. 
Extend the Pochhammer symbols to $\baa \in P$ by 
$$(\baa)_\n=\prod_{\a\in\ck} (\a)_\n^{(\baa,\a)}, \quad (\baa)^\0_\n=\prod_{\a\in\ck} {(\a)^\0_\n}^{(\baa,\a)}.$$
For $\baa, \bbb \in P$, the hypergeometric function is defined by 
$$F(\baa,\bbb;\l)= \frac{1}{1-q} \sum_{\n \in \ck}  \frac{(\baa)_\n}{(\bbb)^\0_\n} \n(\l).$$
It is a function on $\k$ depending on the choice of $\psi$, which takes values in the cyclotomic field $\Q(\mu_{p(q-1)})$. If $\deg \baa = \deg \bbb$, it is independent of $\psi$ and takes values in $\Q(\m_{q-1})$. Note that $F(\baa,\bbb;0)=0$ by definition.   
When $(\bbb,\e)>0$, we also write 
$${}_rF_s\left({\a_1,\dots, \a_r \atop \b_1,\dots, \b_s};\l\right)=F(\a_1+\cdots+\a_r, \e+\b_1+\cdots+\b_s;\l).$$ 
For example, we have \cite[Proposition 2.8]{otsubo}
\begin{equation}\label{g6}
{}_0F_0(\l)=\psi(-\l) \quad (\l\in\k^*). 
\end{equation}
If  $\a \ne \e$, then we have \cite[Corollary 3.4]{otsubo}
\begin{equation}\label{g7}
{}_1F_0\left({\a\atop};\l\right)=\ol\a(1-\l) \quad (\l\in\k^*). 
\end{equation}

A simultaneous shift and an exchange of the parameters result in the following  
\cite[Propositions 2.9, 2.10]{otsubo}: 
\begin{align}
F(\baa,\bbb;\l)&= \frac{(\baa)_\vp}{(\bbb)^\0_\vp} \vp(\l) F(\baa\vp,\bbb\vp;\l), \label{g8}
\\F(\bbb,\baa;\l)&=F(\ol\baa,\ol\bbb; (-1)^{\deg(\baa+\bbb)}\l^{-1}) \quad (\l\in\k^*).\label{g9}
\end{align}
Here, when $\baa=\a_1+\cdots+\a_r$, we write 
$\baa\vp=\a_1\vp+\cdots + \a_r\vp$ and
 $\ol\baa=\ol\a_1+\cdots+\ol\a_r$. 
A cancellation of common factors in the parameters results in the following \cite[Theorem 3.2]{otsubo}: 
\begin{equation}\label{g10}
F(\baa+\bgg,\bbb+\bgg;\l)=q^{(\bgg,\e)}\left(F(\baa,\bbb;\l)+q^{-1}\sum_{\n\in\ck} 
\frac{1-q^{-(\bgg,\n)}}{1-q^{-1}} \frac{(\baa)_{\ol\n}}{(\bbb)^\0_{\ol\n}} \ol\n(\l)
\right). 
\end{equation}

We have the finite analogue of the Euler--Gauss summation formula (cf. \cite[Theorem 4.3]{otsubo}): 
\begin{equation}\label{g11}
{}_2F_1\left({\a,\b\atop \g};1\right)=
\begin{cases}
\dfrac{g^\0(\g)g(\ol{\a\b}\g)}{g^\0(\ol\a\g)g^\0(\ol\b\g)}  & (\a+\b\ne \e+\g), \\
1+q^{\d(\g)}(1-q) 
& (\a+\b=\e+\g).
\end{cases} 
\end{equation}
We have the finite analogue of Kummer's summation formula 
(cf. \cite[Theorem 4.6 (ii)]{otsubo}): if $p\ne 2$, then for any $\a, \b\in\ck$, 
\begin{equation}\label{g12}
{}_2F_1\left({\a^2,\b\atop \a^2\ol\b};-1\right)= \sum_{\a'^2=\a^2} \frac{g^\0(\a^2\ol\b)g(\a')}{g(\a^2)g^\0(\a'\ol\b)}. 
\end{equation}
We have the following finite analogue of Dixon's formula (cf. \cite[Theorem 4.10 (i)]{otsubo}). 
Suppose that $p\ne 2$, $\a^2\ne\b\g$, and $\b+\g\ne \e +\a'$ if $\a'^2=\a^2$. Then 
\begin{equation}\label{g13}
{}_3F_2\left({\a^2,\b,\g \atop \a^2\ol\b,\a^2\ol\g};{1}\right)
=\sum_{\a'^2=\a^2}
\frac{g^\0(\a^2\ol\b)g^\0(\a^2\ol\g)g(\a')g(\a'\ol{\b\g})}
{g(\a^2)g(\a^2\ol{\b\g})g^\0(\a'\ol\b)g^\0(\a'\ol\g)}. 
\end{equation}
See \cite[Section 4]{otsubo} for more formulas for special values ${}_{s+1}F_s(\pm 1)$. 

In Section 4, we will use Jacobi sums in two variables. 
For $\vp$, $\vp'\in\ck$, define 
$$j(\vp,\vp')=-\sum_{x,y\in\ck, x+y=1} \vp(x)\vp'(y).$$
Then we have (cf. \cite[Proposition 2.2 (iv)]{otsubo}) 
\begin{equation}\label{j1}
j(\vp,\vp')= \frac{g(\vp)g(\vp')}{g^\0(\vp\vp')} - \d(\vp)\d(\vp')\frac{(1-q)^2}{q}.
\end{equation}
We will use later 
\begin{equation}\label{j2}
j(\a\n,\ol{\b\n})=\frac{(\b)_{\ol\a}(\a)_\n}{(\e)^\0_{\ol\a}(\b)_\n} \n(-1)+\d(\b\n)(1-q),  
\end{equation}
which follows from \eqref{j1}, examining the exceptional cases where 
$(\ol\a\b+\b+\b\n,\e)\ne 0$. 

%%%
\section{Product formulas}

Let us recall finite analogues of $\eqref{b1}$, $\eqref{b2}$, $\eqref{b6}$, $\eqref{b9}$ (see \cite[Theorems 6.1, 3.14, 6.2]{otsubo}). 
When $p\ne 2$, let $\phi\in\ck$ denote the unique quadratic character, i.e. $\phi^2=\e$, $\phi\ne\e$. 

If $(\a,\b+\e)=0$, then 
\begin{equation}\label{p1}
\psi(\l) {}_1F_1\left({\a\atop \b};\l\right)={}_1F_1\left({\ol\a\b \atop \b};-\l\right).
\end{equation}

If $p\ne 2$ and $\a \ne \e$, then
\begin{equation}\label{p2}
\psi\left(\frac{\l}{2}\right) {}_1F_1\left({\a\atop \a^2};\l\right)= {}_0F_1\left({\atop \a\phi};\frac{\l^2}{16}\right).
\end{equation}

If $(\a+\b,\e+\g)=0$, then
\begin{equation}\label{p6}
\a\b\ol\g(1-\l) {}_2F_1\left({\a,\b \atop \g};\l\right) 
= {}_2F_1\left({\ol\a\g, \ol\b\g \atop \g};\l\right) \quad (\l\ne 1).
\end{equation}

If $p\ne 2$ and $(\a,\e+\b+\b\phi+\b^2)=0$, then 
\begin{equation}\label{p9}
{}_1F_1\left({\a\atop \b^2};\l\right) {}_1F_1\left({\a \atop \b^2};-\l\right) ={}_2F_3\left({\a,\ol\a\b^2 \atop \b^2,\b,\b\phi};{\frac{\l^2}{4}}\right).
\end{equation}

To prove such identities, we compare the Fourier transforms. 
Let $G$ be a finite group. 
For a function $f\colon G \to \C$, its Fourier transform is a function on $\wh G=\Hom(G,\C^*)$, defined by 
$$\wh f(\n) = \sum_{g\in G} f(g) \ol\n(g).$$ 
Then, $f$ is recovered from $\wh f$ by the inverse Fourier transform 
$$f(g)=\frac{1}{\# G} \sum_{\n\in\wh G} \wh f(\n) \n(g).$$ 
Note that the characteristic function of the unit element of $G$ corresponds to the constant function $\wh f\equiv 1$.  

The following is a finite analogue of \eqref{b3}. 
\begin{thm}\label{p3}
If $p \ne 2$ and $(\a^2,\b^2)=(\a^2\b^2,\e)=0$, then 
\begin{align*}
{}_0F_1\left({\atop \a^2};\l\right) {}_0F_1\left({\atop \b^2};\l\right)  
&=
{}_2F_3\left({\a\b,\a\b\phi \atop \a^2,\b^2,\a^2\b^2};4\l\right), 
\\
{}_0F_1\left({\atop \a^2\phi};\l\right) {}_0F_1\left({\atop \b^2\phi};\l\right)  
&=
{}_2F_3\left({\a\b,\a\b\phi \atop \a^2\phi,\b^2\phi,\a^2\b^2};4\l\right). 
\end{align*}
\end{thm}

\begin{proof}
We prove the first formula; the proof of the second formula is parallel. 
Let $f(\l)$ and $g(\l)$ be the left and right hand sides of the formula, respectively. Since $f(0)=g(0)=0$, it suffices to show that $f=g$ as functions on $\k^*$. 
For any $\n\in\ck$, we have by using \eqref{g1} and \eqref{g3}, 
\begin{align*}
\wh f(\n)& =\frac{1}{q-1} \sum_{\m} \frac{1}{(\e+\a^2)^\0_\m} \frac{1}{(\e+\b^2)^\0_{\ol\m\n}}
\\&=\frac{1}{q-1} \frac{1}{(\e+\b^2)^\0_\n} \sum_\m  \frac{1}{(\e+\a^2)^\0_\m} \frac{1}{(\n+\b^2\n)^\0_{\ol\m}}
\\&=\frac{1}{q-1} \frac{1}{(\e+\b^2)^\0_\n} \sum_\m  \frac{(\ol\n+\ol{\b^2\n})_{\m}}{(\e+\a^2)^\0_\m}
\\&= -  \frac{1}{(\e+\b^2)^\0_\n} {}_2F_1\left({\ol\n,\ol{\b^2\n}\atop \a^2};1\right).
\end{align*}
Since $\ol\n+\ol{\b^2\n} \ne \e+\a^2$ by assumption, we can apply \eqref{g11} and 
\begin{align*}
\wh f(\n)& = -  \frac{1}{(\e+\b^2)^\0_\n} \frac{g^\0(\a^2)g(\a^2\b^2\n^2)}{g^\0(\a^2\n)g^\0(\a^2\b^2\n)}
\\&=-\frac{(\a^2\b^2)_{\n^2}}{(\e+\a^2+\b^2+\a^2\b^2)^\0_\n}
\\&=-\frac{(\a\b+\a\b\phi)_\n}{(\e+\a^2+\b^2+\a^2\b^2)^\0_\n} \n(4)=\wh g(\n),
\end{align*}
where we used \eqref{g5} with $n=2$.  Hence we have $f=g$. 
\end{proof}

The following is a finite analogue of \eqref{b4}.
\begin{thm}\label{p4}
If $p \ne 2$, then for any $\a\in\ck$,  
\begin{equation*}
{}_0F_1\left({\atop \a^2};\l\right) {}_0F_1\left({\atop \a^2};-\l\right)  
=
{}_0F_3\left({\atop \a^2,\a,\a\phi};-\frac{\l^2}{4}\right), 
\end{equation*}
\end{thm}

\begin{proof}
Let $f(\l)$, $g(\l)$ be as before. Since both sides are even, it suffices to show $\wh f(\n^2)=\wh g(\n^2)$ for all $\n$. First, similarly as before, we have applying \eqref{g12} and \eqref{g5}, 
\begin{align*}
\wh f(\n^2)&=-\frac{1}{(\e+\a^2)^\0_{\n^2}} {}_2F_1\left({\ol{\n^2},\ol{\a^2\n^2} \atop \a^2};-1\right).
\\&=-\frac{1}{(\e+\a^2)^\0_{\n^2}} \sum_{\m^2=\n^2} \frac{g^\0(\a^2) g(\ol\m)}{g(\ol{\n^2})g^\0(\a^2\m)}
\\&=-\frac{1}{(\e+\a^2)^\0_{\n^2}} \sum_{\m^2=\n^2} \frac{g^\0(\a^2) g^\0(\n^2)}{g^\0(\m)g^\0(\a^2\m)} \m(-1)
\\&= - \sum_{\m^2=\n^2} \frac{1}{(\e+\a^2)^\0_\m (\a^2)^\0_{\m^2}}\m(-1)
\\&=- \sum_{\m^2=\n^2} \frac{1}{(\e+\a^2+\a+\a\phi)^\0_\m} \m\left(-\frac{1}{4}\right). 
\end{align*}
On the other hand, 
\begin{align*}
\wh g(\n^2)&=\frac{1}{1-q} \sum_{\m\in\ck}\frac{1}{(\e+\a^2+\a+\a\phi)^\0_\m} \sum_{\l\in\k^*}  \m\left(-\frac{\l^2}{4}\right) \ol{\n^2}(\l)
\\&=\frac{1}{1-q} \sum_{\m\in\ck}\frac{1}{(\e+\a^2+\a+\a\phi)^\0_\m} \m\left(-\frac{1}{4}\right)
 \sum_{\l\in\k^*} \m^2\ol{\n^2}(\l). 
\end{align*}
By the orthogonality of characters, we have $\wh f(\n^2) =\wh g(\n^2)$, hence the theorem. 
\end{proof}

The following are finite analogues of \eqref{b5}, \eqref{b7}, \eqref{b8}, \eqref{b10}, \eqref{b11}, respectively. 
\begin{cor}Suppose that $p\ne 2$. 
\begin{enumerate}
\item
For any $\a\in\ck$,  we have
\begin{equation*}\label{p5}
{}_0F_1\left({\atop \a^2};\l\right) {}_0F_1\left({\atop \ol{\a^2}};-\l\right) 
= q^{\d(\a)} {}_0F_3\left({\atop \phi,\a\phi,\ol\a\phi};-\frac{\l^2}{4}\right). 
\end{equation*}
\item If $(\a^2+\b^2+\a^2\b^2,\e)=0$, then 
\begin{equation*}\label{p7}
{}_2F_0\left({\a^2,\b^2\atop};\l\right) {}_2F_0\left({\a^2,\b^2\atop};-\l\right) 
={}_4F_1\left({\a^2,\b^2,\a\b,\a\b\phi \atop \a^2\b^2};4\l^2\right).
\end{equation*}

\item If $(\a^2+\b^2+\a^2\b^2,\e)=(\a^2,\b^2)=0$, then 
\begin{equation*}\label{p8}
{}_2F_0\left({\a^2, \ol{\a^2}\atop };\l\right) {}_2F_0\left({\b^2, \ol{\b^2}\atop};-\l\right)  
={}_4F_1\left({\a\ol\b\phi, \ol\a\b\phi, \a\b, \ol{\a\b}  \atop \phi};4\l^2\right). 
\end{equation*}

\item If $(\a,\e+\b^2)=(\a^2,\b^2)=0$, then 
\begin{equation*}\label{p10}
{}_1F_1\left({\a\atop \b^2};\l\right) {}_1F_1\left({\a\ol{\b^2} \atop \ol{\b^2}};-\l\right)
={}_2F_3\left({\a\ol\b\phi,\ol\a\b\phi \atop \phi, \b\phi, \ol\b\phi};\frac{\l^2}{4}\right).
\end{equation*}

\item If $(\a^2+\b^2+\a^2\b^2,\e)=(\a^2,\b^2)=0$, then
\begin{equation*}\label{p11.1}
{}_1F_1\left({\a^2\atop \a^4};\l\right){}_1F_1\left({\b^2\atop \b^4};-\l\right)
={}_2F_3\left({\a\b,\a\b\phi \atop \a^2\phi,\b^2\phi,\a^2\b^2};\frac{\l^2}{4}\right).
\end{equation*}
If $(\a^2\phi+\b^2\phi+\a^2\b^2,\e)=(\a^2,\b^2)=0$, then
\begin{equation*}\label{p11.2}
{}_1F_1\left({\a^2\phi \atop \a^4};\l\right){}_1F_1\left({\b^2\phi \atop \b^4};-\l\right)
={}_2F_3\left({\a\b,\a\b\phi \atop \a^2,\b^2,\a^2\b^2};\frac{\l^2}{4}\right). 
\end{equation*}
\end{enumerate}
\end{cor}

\begin{proof}Since all the formulas are obviously true for $\l=0$, we assume that $\l\ne 0$. 

(i) By \eqref{g8}, \eqref{g2} and \eqref{g4} with $n=2$, we have 
$${}_0F_1\left({\atop \ol{\a^2}};-\l\right)= \frac{g^\0(\ol{\a^2})}{g^\0(\a^2)} \a^2(\l){}_0F_1\left({\atop \a^2};-\l\right), 
$$
and 
\begin{align*}
{}_0F_3\left({\atop \phi,\a\phi,\ol\a\phi};-\frac{\l^2}{4}\right)
&= 
\frac{g^\0(\phi)g^\0(\ol\a\phi)}{g^\0(\a)g^\0(\a^2)}
\a\phi\left(-\frac{\l^2}{4}\right)
{}_0F_3\left({\atop \a^2,\a,\a\phi};-\frac{\l^2}{4}\right)
\\&=q^{-\d(\a)} \frac{g^\0(\ol{\a^2})}{g^\0(\a^2)}\a^2(\l) {}_0F_3\left({\atop \a^2,\a,\a\phi};-\frac{\l^2}{4}\right).
\end{align*}
Hence the formula reduces to Theorem \ref{p4}. 

(ii) By \eqref{g8}, \eqref{g2} and \eqref{g9},
\begin{align*}
{}_2F_0\left({\a^2,\b^2\atop};\l\right)
&= \frac{g(\ol{\a^2}\b^2)}{g(\b^2)} \ol{\a^2}(\l) F(\e+\ol{\a^2}\b^2, \ol{\a^2}; \l) 
\\&=\frac{g(\ol{\a^2}\b^2)}{g(\b^2)} \a^2(\l^{-1}) {}_1F_1\left({\a^2 \atop \a^2\ol{\b^2}};-\l^{-1}\right). 
\end{align*}
Similarly, using \eqref{g4}, we have
$${}_4F_1\left({\a^2,\b^2,\a\b,\a\b\phi \atop \a^2\b^2};4\l^2\right)
=\frac{g(\ol{\a^2}\b^2)^2}{g(\b^2)^2} \a^4(\l^{-1}) {}_2F_3\left({\a^2,\ol{\b^2} \atop \a^2\ol{\b^2}, \a\ol\b, \a\ol\b\phi};\frac{\l^{-2}}{4}\right). 
$$
Hence the formula reduces to \eqref{p9}. 

(iii) This is equivalent to the first statement of (v), which we prove below. 

(iv) This reduces easily to \eqref{p9}, using \eqref{g8}. 

(v) By \eqref{p2}, the left hand side of the first formula equals 
$${}_0F_1\left({\atop \a^2\phi};\frac{\l^2}{16}\right) {}_0F_1\left({\atop \b^2\phi};\frac{\l^2}{16}\right),$$
and the formula follows from the second formula of Theorem \ref{p3}. 
The second formula follows similarly from the first formula of Theorem \ref{p3}. 
\end{proof}

Finally, the following is a finite analogue of \eqref{b12}. 
\begin{thm}Suppose that $6 \mid q-1$ and let $\r\in\ck$ be a cubic character, i.e. $\r^3=\e$, $\r\ne \e$. 
If $(\a^6,\b^{12})=(\a^{12},\b^6)=(\a^6\b^6,\e)=0$, then 
\begin{equation*}\label{p12}\begin{split}
& {}_0F_2\left({\atop \a^6,\b^6};\l\right){}_0F_2\left({\atop \a^6,\b^6};-\l\right)
\\& ={}_3F_8\left({\a^2\b^2,\a^2\b^2\r,\a^2\b^2\ol\r \atop\a^6,\b^6,\a^3,\a^3\phi,\b^3,\b^3\phi, \a^3\b^3, \a^3\b^3\phi};-\frac{27\l^2}{64}\right).
\end{split}\end{equation*}
\end{thm}

\begin{proof}
Let $f(\l)$ and $g(\l)$ be as before; these are even functions. 
First, similarly as before, 
\begin{align*}
\wh f(\n^2) &= -\frac{1}{(\e+\a^6+\b^6)^\0_{\n^2} }
{}_3F_2\left({\ol{\n^2},\ol{\a^6\n^2},\ol{\b^6\n^2} \atop \a^6,\b^6};1\right). 
\end{align*}
We show that 
$${}_3F_2\left({\ol{\n^2},\ol{\a^6\n^2},\ol{\b^6\n^2} \atop \a^6,\b^6};1\right)
=\frac{g^\0(\a^6)g^\0(\b^6)}{g(\ol{\n^2})g^\0(\a^6\b^6\n^2)} 
\sum_{\m^2=\n^2} \frac{g(\ol\m)g(\a^6\b^6\m^3)}{g^\0(\a^6\m)g^\0(\b^6\m)}.
$$
If $\n^2\ne \ol{\a^6\b^6}$, this follows by \eqref{g13}. 
Suppose that  $\n^2= \ol{\a^6\b^6}$. 
Then the left hand side becomes, using \eqref{g10}, \eqref{g7} and \eqref{g2}, 
$${}_3F_2\left({\a^6\b^6,\a^6,\b^6\atop \a^6,\b^6};1\right)
=
\begin{cases}
2\dfrac{g^\0(\a^6)g^\0(\b^6)}{g(\a^6\b^6)q} & (\a^6\ne\b^6), \\
\dfrac{g^\0(\a^6)^2}{g(\a^{12})q}(1+q^{-1}) & (\a^6=\b^6). 
\end{cases}
$$
The right hand side becomes by \eqref{g2},  
$$\frac{g^\0(\a^6)g^\0(\b^6)}{g(\a^6\b^6)q}\sum_{\m^2=\n^2} \frac{g(\ol\m)g(\m)}{g^\0(a^6\m)g^\0(\ol{\a^6\m})}
=\frac{g^\0(\a^6)g^\0(\b^6)}{g(\a^6\b^6)q}(1+q^{-(\a^6,\b^6)}). 
$$
Note that $\m \ne \e$ by the assumption $(\a^6\b^6,\e)=0$. 
Hence we have the equality as above, and obtain
\begin{align*}
\wh f(\n^2) &= -\frac{1}{(\e+\a^6+\b^6+\a^6\b^6)^\0_{\n^2} (\e)_{\ol{\n^2}}}\sum_{\m^2=\n^2} 
\frac{(\e)_{\ol\m}(\a^6\b^6)_{\m^3}}{(\a^6+\b^6)^\0_\m}
\\&= -\frac{1}{(\a^6+\b^6+\a^6\b^6)^\0_{\n^2} }\sum_{\m^2=\n^2} 
\frac{(\a^6\b^6)_{\m^3}}{(\e+\a^6+\b^6)^\0_\m}\m(-1). 
\end{align*}
On the other hand, by \eqref{g4} with $n=2$ and $3$, 
\begin{align*}
\wh{g}(\n^2)&= -\sum_{\m^2=\n^2} \frac{(\a^2\b^2+\a^2\b^2\r+\a^2\b^2\ol\r)_{\m}}{(\e+\a^6+\b^6+\a^3+\a^3\phi+\b^3+\b^3\phi+\a^3\b^3+\a^3\b^3\phi)^\0_{\m}} \m\left(-\frac{27}{64}\right)
\\&=-\sum_{\m^2=\n^2} \frac{(\a^6\b^6)_{\m^3}}{(\e+\a^6+\b^6)^\0_\m (\a^6+\b^6+\a^6\b^6)^\0_{\m^2}} \m(-1).
\end{align*}
Hence the theorem follows. 
\end{proof}

%%%
\section{Barman--Tripathi formula}

We give a short proof of the finite analogue of Bailey's formula \eqref{b14} due to Barman--Tripathi \cite{tripathi-barman}. 
As well as the original proof, our proof intervenes the finite analogue of Appell's hypergeometric function $F_4$ in two variables defined by 
$$
F_4(\a,\b;\g,\g';\l,\l')=\frac{1}{(1-q)^2}\sum_{\n,\n' \in\ck} \frac{(\a+\b)_{\n\n'}}{(\e+\g)^\0_\n(\e+\g')^\0_{\n'}} \n(\l)\n'(\l'). 
$$
The following is a finite analogue of a complex formula (cf. \cite[9.6]{bailey-book}). 
Let $\d\colon \k \to \C$ denote the characteristic function of $0$, i.e. $\d(0)=1$ and $\d(\l)=0$ if $\l\ne 0$. 

\begin{thm}[{\cite[Theorem 1.2]{tripathi-barman1}}]\label{t1} 
Suppose that $\a\b=\g\g'$ and $(\a+\b,\e+\g)=0$ (so $(\a+\b,\g')=0$). 
Then, if $x\ne 1$ and $y \ne 1$, 
\begin{align*}
&{}_2F_1\left({\a,\b\atop \g};\frac{x}{x-1}\right) {}_2F_1\left({\a,\b\atop \g'};\frac{y}{y-1}\right) 
-\d(1-xy)C \ol\b\g(y) \a(1-x)\b(1-y)
\\&= F_4\left({\a,\b;\g,\g'};\frac{-x}{(1-x)(1-y)},\frac{-y}{(1-x)(1-y)}\right)
\end{align*}
with
$$C=\frac{g^\0(\g)g^\0(\g')}{g(\a)g(\b)}\b\ol\g(-1).$$
\end{thm}

\begin{proof}
Put, for $x \ne 1$, $y \ne 1$,  
$$f(x,y)=\ol\a(1-x)\ol\b(1-y) F_4\left({\a,\b;\g,\g'};\frac{-x}{(1-x)(1-y)},\frac{-y}{(1-x)(1-y)}\right),$$ 
and extend it by zero to a function on the group $(\k^*)^2$. 
Then we have by \eqref{j2}, 
\begin{align*}
\wh f(\n,\n')&=\frac{1}{(1-q)^2} \sum_{\m,\m'} \frac{(\a+\b)_{\m\m'}}{(\e+\g)^\0_\m(\e+\g')^\0_{\m'}} \m\m'(-1) j(\m\ol\n,\ol{\a\m\m'}) j(\m'\ol{\n'},\ol{\b\m\m'})
\\&= \frac{1}{(1-q)^2} \sum_{\m,\m'} 
\frac{(\a+\b)_{\m\m'}}{(\e+\g)^\0_\m(\e+\g')^\0_{\m'}} 
\left(\frac{(\a\m')_\n(\ol\n)_\m}{(\e)^\0_\n(\a\m')_\m} +\d(\a\m\m')(1-q)\m(-1)\right)
\\&\phantom{ \frac{1}{(1-q)^2} \sum_{\m,\m'} \frac{(\a+\b)_{\m\m'}}{(\e+\g)^\0_\m(\e+\g')^\0_{\m'}}} \times
\left(\frac{(\b\m)_{\n'}(\ol{\n'})_{\m'}}{(\e)^\0_{\n'}(\b\m)_{\m'}} +\d(\b\m\m')(1-q)\m'(-1)\right)
\\&
=\frac{1}{(1-q)^2}\sum_{\m,\m'}\frac{(\a)_\n(\b)_{\n'}}{(\e)^\0_\n(\e)^\0_{\n'}} \frac{(\b\n'+\ol\n)_\m(\a\n+\ol{\n'})_{\m'}}{(\e+\g)^\0_\m(\e+\g')^\0_{\m'}}+R_1(\n)+R_2(\n')+R_3
\\&= M(\n,\n')+R_1(\n)+R_2(\n')+R_3, 
\end{align*}
where  
$$M(\n,\n'):=\frac{(\a)_\n(\b)_{\n'}}{(\e)^\0_\n(\e)^\0_{\n'}} 
{}_2F_1\left({\ol\n, \b\n' \atop \g};1\right)
{}_2F_1\left({\a\n,\ol{\n'}\atop \g'};1\right),$$
$R_1$ and $R_2$ are functions on $\ck$, and $R_3$ is a constant. 
We compare $M(\n,\n')$ with
$$N(\n,\n'):= \frac{(\a+\ol\b\g)_\n}{(\e+\g)^\0_\n}  \frac{(\b+\ol\a\g')_{\n'}}{(\e+\g')^\0_{\n'}}.
$$
First, if $\ol\n+\b\n' \ne \e+\g$ and $\a\n+\ol{\n'}\ne\e+\g'$, then by \eqref{g11},  
\begin{align*}
{}_2F_1\left({\ol\n, \b\n' \atop \g};1\right)
&=\frac{(\ol\a\g')_{\n'}}{(\g)^\0_\n}\frac{g(\ol\b\g\n\ol{\n'})}{g^\0(\ol\b\g)}\n'(-1), 
\\
{}_2F_1\left({\a\n,\ol{\n'} \atop \g'};1\right)
&=\frac{(\ol\b\g)_{\n}}{(\g')^\0_{\n'}}\frac{g(\ol\a\g' \ol\n\n')}{g^\0(\ol\a\g')}\n(-1),
\end{align*}
and it follows $M(\n,\n')=q^{-\d(\ol\b\g\n\ol{\n'})}N(\n,\n')$ by using \eqref{g2}. 
If moreover $\ol\b\g\n\ol{\n'}=\e$, then one computes $N(\n,\n')=qC$. 
Hence we have
\begin{equation}\label{e4.1}
M(\n,\n')=N(\n,\n')+\d(\ol\b\g\n\ol{\n'})(1-q)C.
\end{equation}
On the other hand, if $\b\n'+\ol\n=\e+\g$ (then $\a\n+\ol{\n'}\ne \e+\g'$), then   
$$M(\n,\n')=(q^{-\d(\g)}+1-q) C, \quad N(\n,\n')=q^{-\d(\g)}C.$$ 
Hence $M(\n,\n')=N(\n,\n')+(1-q)C$, and the same holds when $\a\n+\ol{\n'}=\e+\g'$. 
Noting that $\ol\b\g\n\ol{\n'}=\e$ in these cases, the relation \eqref{e4.1} always holds. 

Now, the inverse Fourier transform of $\d(\ol\b\g\n\ol{\n'})$ is 
$$\frac{1}{(q-1)^2} \sum_{\n\in\ck} \n(x) \ol\b\g\n(y)=\frac{1}{(q-1)^2} \sum_\n \n(xy) \ol\b\g(y)
=\frac{1}{q-1} \d(1-xy) \ol\b\g(y).$$
The inverse Fourier transform,  as functions on $(\ck)^2$,  of $R_1(\n)$ (resp. of $R_2(\n')$, of $R_3$) 
vanishes if $y \ne 1$ (resp. if $x\ne 1$, if $(x,y)\ne(1,1)$). 
Hence we have
\begin{align*}
f(x,y)&=\frac{1}{(q-1)^2}\sum_{\n,\n'\in\ck} N(\n,\n')\n(x)\n'(y) -\d(1-xy) C \ol\b\g(y)
\\&= {}_2F_1\left({\a,\ol\b\g\atop \g};x\right){}_2F_1\left({\b,\ol\a\g'\atop \g'};y\right)-\d(1-xy) C\ol\b\g(y)
\end{align*}
if $x\ne 1$ and $y\ne 1$. 
By the finite analogue of Pfaff's  transformation formula 
$$ {}_2F_1\left({\a,\ol\b\g\atop \g};x\right) = \ol\a(1-x) {}_2F_1\left({\a,\b \atop \g};\frac{x}{x-1}\right) \quad (x \ne 1)$$
(cf. \cite[Theorem 3.14 (ii)]{otsubo}), the theorem follows. 
\end{proof}

The following is a finite analogue of \eqref{b14}. 
\begin{cor}[{\cite[Theorem 1.1]{tripathi-barman}}]\label{b-t} 
Suppose that $p\ne 2$, $\a^2\b^2=\g\g'$ and $(\a^2+\b^2,\e+\g)=(\a^2\b^2,\e+\g^2)=0$.
Then, 
\begin{align*}
&{}_2F_1\left({\a^2,\b^2\atop \g};\l\right) {}_2F_1\left({\a^2,\b^2\atop \g'};\l\right)
-\d(1-2\l)\frac{g^\0(\g)g^\0(\g')}{g(\a^2)g(\b^2)}  \a\b(4)
\\&={}_4F_3\left({\a^2,\b^2, \a\b, \a\b\phi \atop \a^2\b^2, \g,\g'};4\l(1-\l)\right)+\d(1-\l). 
\end{align*}
\end{cor}

\begin{proof}
When $\l \ne 1$, this follows immediately from Theorem \ref{t1} by
replacing $\a$ (resp. $\b$) with $\a^2$ (resp. $\b^2$), 
letting $x=y=\frac{\l}{\l-1}$, 
and applying the lemma below. 
The formula for $\l=1$ follows by using \eqref{g11}. 
\end{proof}

\begin{lem}
Suppose that $p\ne 2$, $\a^2\b^2=\g\g'$ and $(\g\g',\e)=(\g,\g')=0$. 
Then,  
$$F_4(\a^2,\b^2;\g,\g';x,x)={}_4F_3\left({\a^2,\b^2,\a\b,\a\b\phi \atop \a^2\b^2,\g,\g'};4x\right).$$
\end{lem}

\begin{proof}
The left hand side is written as 
\begin{align*}
&\frac{1}{(1-q)^2} \sum_{\n,\n'} \frac{(\a^2+\b^2)_{\n\n'}}{(\e+\g)^\0_\n(\e+\g')^\0_{\n'}} \n\n'(x)
\\&= \frac{1}{(1-q)^2} \sum_\m (\a^2+\b^2)_\m \m(x) \sum_\n \frac{1}{(\e+\g)^\0_\n(\e+\g')^\0_{\m\ol\n}}
\\&=\frac{1}{1-q} \sum_\m \frac{(\a^2+\b^2)_\m}{(\e+\g')^\0_\m} \m(x) {}_2F_1\left({\ol\m,\ol{\g'\m}\atop \g};1\right), 
\end{align*}
where we used \eqref{g1} and \eqref{g3}.
Since $\ol\m+\ol{\g'\m} \ne \e+\g$ by assumption, we have by \eqref{g11} and \eqref{g4},  
$${}_2F_1\left({\ol\m,\ol{\g'\m}\atop \g};1\right) =\frac{(\a^2\b^2)_{\m^2}}{(\a^2\b^2+\g)^\0_\m}
 =\frac{(\a\b+\a\b\phi)_\m}{(\a^2\b^2+\g)^\0_\m} \m(4).$$
Hence the lemma follows. 
\end{proof}

%%%%%
\section*{Acknowledgements}
This paper is largely based on the second author's master's thesis \cite{senoue}. 
We would like to thank Ryojun Ito, Akio Nakagawa and Yusuke Nemoto for helpful discussions. 
The first author is supported by JSPS Grant-in-Aid for Scientific Research: 18K03234.

%%%%%
%\section*{Data Availability}
%Data sharing not applicable to this article as no datasets were generated or analysed during the current study. 

%%%%%%%%%%%%%%%%%


\begin{thebibliography}{AA}

\bibitem{bailey1} W. N. Bailey, Products of generalized hypergeometric functions, Proc. London Math. Soc. Ser. 2, {\bf 28} (1928), 242--254. 

\bibitem{bailey2} W. N. Bailey, Some theorems concerning products of hypergeometric series, Proc. London Math. Soc. Ser. 2, {\bf 38} (1935), 377--384. 

\bibitem{bailey-book} W. N. Bailey, {\it Generalized Hypergeometric Series}, Cambridge Univ. Press, 1935. 

\bibitem{evans2} R. Evans and J. Greene, 
Clausen's theorem and hypergeometric functions over finite fields, 
Finite Fields and Their Applications {\bf 15} (2009), 97--109. 

\bibitem{fuselieretal}
J. Fuselier, L. Long, R. Ramakrishna, H. Swisher and F.-T. Tu, Hypergeometric functions over finite fields, 
Mem. Amer. Math. Soc. (to appear). 

\bibitem{greene} J. Greene, Hypergeometric functions over finite fields, Trans. Amer. Math. Soc., {\bf 301} (1987), 77--101.

\bibitem{katz} N. M. Katz, {\it Exponential Sums and Differential Equations}, Annals of Math. Studies {\bf 124}, Princeton, 1990.

\bibitem{koblitz} N. Koblitz, The number of points on certain families of hypersurfaces over finite fields, Compositio Math. {\bf 48} (1983), 3--23.

\bibitem{mccarthy} D. McCarthy, Transformations of well-poised hypergeometric functions over finite fields, Finite Fields and Their Applications {\bf 18} (2012), 1133--1147. 
\bibitem{otsubo} N. Otsubo, Hypergeometric functions over finite fields, \href{https://arxiv.org/abs/2108.06754}{arXiv:2108.06754}. 

\bibitem{senoue} T. Senoue, Product formulas for hypergeometric functions over finite fields (in Japanese), Master's Thesis, Chiba University, March 2022. 

\bibitem{tripathi-barman1} M. Tripathi and R. Barman, 
A finite field analogue of the Appell series $F_4$, Res. Number Theory {\bf 4}, 35 (2018). 

\bibitem{tripathi-barman} M. Tripathi and R. Barman, Certain product formulas and values of Gaussian hypergeometric series, Res. Number Theory {\bf 6}, 26 (2020). 

\end{thebibliography}
\end{document}